%% file: main.tex
\documentclass[
  letterpaper,
  DIV=11,
  numbers=noendperiod
]{scrartcl}

\input{preamble}
\input{macros}
\input{theorems}
\usepackage{amsmath}

\DeclareMathOperator*{\minimize}{minimize}

\title{Adaptive Normalized Gradient Method for $(L_0,L_1)$-smooth
Problems: Estimating Problem Parameters as a Region}
\author[1]{Rikiya Okazaki}
\author[2]{Shotaro Yagishita}
\author[3,1]{Hideitsu Hino}

\affil[1]{Graduate Institute for Advanced Studies, SOKENDAI}
\affil[2]{Graduate School of Informatics, Kyoto University}
\affil[3]{The Institute of Statistical Mathematics}

\date{\today}
\begin{document}
\maketitle
\begin{abstract} 
We propose a parameter-free and efficient gradient descent algorithm for $(L_0,L_1)$-smooth optimization. This notion of smoothness, introduced by \citet{zhang2020gradientclippingacceleratestraining}, allows the smoothness bound to grow with the gradient norm.
Similar to cutting-plane localization, our method uses a sequence of observations to maintain a region of plausible values for $(L_0,L_1)$ and uses this region to select step sizes. 
In the convex setting, our method attains the standard iteration complexity bound $\tilde{\mathcal{O}}\left(\left(L_0/\epsilon + L_1^2\right)\|x^0-x^*\|^2\right)$ for finding an $\epsilon$-optimal solution, 
while in the nonconvex setting it finds an $\epsilon$-stationary point within $\tilde{\mathcal{O}}\left(\left(L_0/\epsilon^2 + L_1/\epsilon\right)(f(x^0)-f^*)\right)$ iterations.
To the best of our knowledge, this is the first work to establish, for a single parameter-free algorithm, the best-known iteration-complexity guarantee in the nonconvex setting while also providing the above iteration-complexity guarantee in the convex setting.
Experiments on logistic regression and real phase retrieval demonstrate that our method can exploit $(L_0,L_1)$-smoothness and perform efficiently on nonconvex problems.
\end{abstract}

\floatname{algorithm}{Algorithm}

\renewcommand*\contentsname{Table of contents}
{
\setcounter{tocdepth}{1}
\tableofcontents
}

\section{Introduction}\label{sec-introduction}
We consider the following unconstrained optimization problem:
\[
\minimize_{x \in \R^d} f(x).
\]
Recent advances in computing have made optimization in high-dimensional spaces increasingly common, particularly in neural network training.
In high-dimensional problems, computing derivatives of order two or higher can be computationally expensive, and storing them can require substantial memory. First-order methods, which use only function values and gradients, are therefore widely used, and improving their efficiency is of great importance.
Among first-order methods, gradient descent method  iteratively updates the solution in the direction of the negative gradient of the objective:
\begin{equation}\protect\phantomsection\label{eq-gd-def}{
x^{k+1} = x^k - \eta_k \nabla f(x^k),
}\end{equation}
where $\eta_k$ is the step size of the update.

This work focuses on the step-size selection problem for gradient descent.
In particular, we propose the normalized gradient method with region-based adaptation (\hyperref[algo-ngm-ra]{NGM-RA}), an efficient algorithm for optimization problems whose objective $f$ is $(L_0,L_1)$-smooth.
We define $(L_0,L_1)$-smoothness of $f$ as follows.
\begin{definition}[$(L_0,L_1)$-smoothness
\citep{zhang2020gradientclippingacceleratestraining}]\protect\hypertarget{def-l0l1sm}{}\label{def-l0l1sm}

A function $f:\R^d \to \R$ is $(L_0,L_1)$-smooth if it satisfies
\begin{equation}\protect\phantomsection\label{eq-l0l1-def}{
\forall x \in \R^d, \; \|\nabla^2 f(x)\| \leq L_0 + L_1 \|\nabla f(x)\|.
}\end{equation}

\end{definition}
Here, for a vector $v \in \R^d$, $\|v\|$ denotes its Euclidean norm, and for a matrix $A\in \R^{d \times d}$, $\|A\|$ denotes its operator norm.
This notion generalizes $L$-smoothness, defined below, since every $L$-smooth function is $(L,0)$-smooth.
\begin{definition}[$L$-smoothness]\protect\hypertarget{def-lsm}{}\label{def-lsm}

A function $f:\R^d \to \R$ is $L$-smooth if it satisfies

\[
\forall x \in \R^d, \; \|\nabla^2 f(x)\| \leq L.
\]

\end{definition}
The $(L_0,L_1)$-smoothness condition for objective functions was introduced by \citet{zhang2020gradientclippingacceleratestraining} and has attracted interest for the following reasons.
First, there exist functions that are $(L_0,L_1)$-smooth but not $L$-smooth.
Examples include univariate polynomial~\citep{zhang2020gradientclippingacceleratestraining}, power of norm and exponential functions~\citep{gorbunov2025methods}.
Moreover, \citet{zhang2020gradientclippingacceleratestraining} and \citet{crawshaw2022robustness} reported that objectives arising from modern neural network architectures such as LSTMs~\citep{gers1999lstm} and Transformers~\citep{vaswani2017attention} exhibit $(L_0,L_1)$-smoothness.
Second, even when an objective satisfies both smoothness conditions, an algorithm tailored to $(L_0,L_1)$-smoothness can exhibit better convergence behavior.
This advantage can be seen theoretically in simple examples. For instance, the logistic objective $f(x)=\log(1+\exp\{a^Tx\})$ is both $\|a\|^2$-smooth and $(0,\|a\|)$-smooth~\citep{gorbunov2025methods}.
As discussed below, replacing gradient descent with a fixed step size by normalized gradient descent improves the iteration bound for reaching an $\epsilon$-stationary point from $\mathcal{O}(\|a\|^2/\epsilon^2)$ to $\mathcal{O}(\|a\|/\epsilon)$.

\citet{zhang2020gradientclippingacceleratestraining} pointed out that clipped gradient descent (Clip.GD) and normalized gradient descent (Norm.GD) are effective for $(L_0,L_1)$-smooth problems.
Moreover, for appropriately chosen $\alpha$ and $\beta$, setting the step size $\eta_k$ in Equation~\ref{eq-gd-def} for Clip.GD and Norm.GD, respectively, to
\[
\begin{aligned}
\eta_k
&= \frac{1}{\alpha}
\min\left\{\frac{1}{L_0},
\frac{1}{\beta \|\nabla f(x^k)\|}\right\}
&&(\mathrm{Clip.GD}),\\
\eta_k
&= \frac{1}{\alpha}\frac{1}{L_0+\beta \|\nabla f(x^k)\|}
&&(\mathrm{Norm.GD}),
\end{aligned}
\]
yields an iteration complexity of
$\mathcal{O}\left((L_0/\epsilon^2+L_1/\epsilon) (f(x^0)-f^*)\right)$
for finding an $\epsilon$-stationary point~\citep{vankov2025optimizingl0l1smooth}.
The leading term
$\mathcal{O}\left(L_0(f(x^0)-f^*)/\epsilon^2\right)$
matches the iteration-complexity lower bound for $L_0$-smooth problems~\citep{carmon2019lowerbounds}.
Moreover, according to a recent result by~\citep{liang2026complexityfindingstationarypoints}, the dependence on $L_1$ is also optimal.

However, problem parameters characterizing properties of the objective function, such as $(L_0,L_1)$, 
are global quantities that can in principle be determined only by examining the entire domain $x\in\mathbb{R}^d$.
Therefore, assuming that such parameters are available to the algorithm 
is often unrealistic in practice.
Algorithms that do not require such problem parameters are referred to as parameter-free algorithms and are of practical importance.
In particular, there has been substantial recent interest in parameter-free algorithms 
that attain, up to negligible terms, the same iteration complexity as algorithms with known problem parameters.
Parameter-free algorithms also have the advantage 
that their complexity can automatically depend on the infimum of the parameter values valid only along the generated iterates, 
potentially making them more efficient than methods using some specific parameter values.

To the best of our knowledge, however, no prior work has established the above iteration-complexity guarantee for a parameter-free algorithm in the nonconvex setting.
Since many practical optimization problems, including the training of neural networks, involve nonconvex objective functions, theoretical analysis in this setting is important.
For $L$-smooth problems, parameter-free methods that adaptively estimate the smoothness parameter have been actively studied since \citet{malitsky2020adaptivegradientdescentdescent}. 
More recently, this line of approach has been extended to the nonconvex setting~\citep{lan2024projectedgradientmethodsnonconvex, yagishita2025simplelsfree}. 
In particular, these methods use $\max_k M_k$ as an estimate of the smoothness parameter for step-size selection, where $M_k$ is a lower bound on the smoothness parameter inferred at iteration $k$.
While a simple convergence analysis of this approach has been provided by \citet{yagishita2025simplelsfree}, it relies on the parameter being scalar and does not directly extend to settings with multiple coupled unknown parameters, such as $(L_0,L_1)$-smoothness.
Motivated by the above observations, the central question we address is the following.

\emph{
Can scalar parameter-free adaptation be extended to multiple coupled parameters? If so, can this extension yield a parameter-free algorithm with the best-known convergence guarantees for nonconvex $(L_0,L_1)$-smooth optimization?
}

We answer this question affirmatively by proposing \hyperref[algo-ngm-ra]{NGM-RA}, which retains the standard convergence guarantee when $f$ is convex while also achieving the best-known convergence guarantees in the nonconvex setting.
The key idea is to reinterpret $\max_k M_k$ as progressively ruling out curvature values that are inconsistent with the observations and selecting the smallest curvature value that remains plausible.
This perspective suggests maintaining a region of plausible parameter values and selecting the most aggressive parameters within that region. 
We show that this strategy can recover the same convergence guarantees as algorithms that know the parameters in advance.

The paper is organized as follows.
In the rest of Section 1, we summarize our contributions and review related work.
In Section 2, we introduce our proposed algorithm, \hyperref[algo-ngm-ra]{NGM-RA}, and describe its details.
In Section 3, we establish upper bounds on the iteration complexity of NGM-RA in both the nonconvex and convex settings and provide an overview of the proofs.
Complete proofs are deferred to the Appendix~\ref{sec-add-proof}.
In Section 4, we evaluate the effectiveness of NGM-RA through numerical experiments.
Finally, we conclude the paper and discuss future directions.

\subsection{Our Contributions}\label{our-contributions}

Our contributions are summarized as follows.

\begin{itemize}
    \item We generalize a parameter-free approach based on parameter estimation and step-size selection from a single smoothness parameter to multiple coupled smoothness parameters. This framework allows us to leverage the convergence analysis of existing parameter-dependent algorithms, leading to a simple analysis.
    \item Building on this idea, we introduce a first-order method, \hyperref[algo-ngm-ra]{NGM-RA}. We show that NGM-RA attains the best-known iteration-complexity guarantee for nonconvex $(L_0,L_1)$-smooth optimization while retaining the standard guarantee in the convex setting\footnote{\citet{bilel2024complexities} claims the same iteration-complexity upper bound for the nonconvex setting for memory Armijo in Theorem 2.1; however, the proof does not account for the case where $\tilde{\eta}_k \leq f_1\eta_k$ does not hold. Furthermore, \citet{bilel2024complexities} does not provide an analysis for the convex setting.}, as discussed in detail in Subsection~\ref{sec-related-work}. 
\end{itemize}

\subsection{Related Work}\label{sec-related-work}

This paper studies parameter-free optimization for $(L_0,L_1)$-smooth problems.
In the following, we review related work on parameter-free methods 
and $(L_0,L_1)$-smooth optimization, respectively.
For ease of comparison, let $\Delta \defeq f(x^0)-f^*$ and $R \defeq \|x^0-x^*\|$, where $x^*$ is an optimal solution.

\paragraph*{Parameter-Free Methods for $L$-smooth Problems}

Since finding a global optimum of a nonconvex optimization problem is NP-hard in general~\citep{murty1987npcomplete},
the iteration complexity for finding an $\epsilon$-stationary point, 
i.e., a point $x\in\mathbb{R}^d$ satisfying $\|\nabla f(x)\|\leq\epsilon$, is commonly studied in the nonconvex setting.
For first-order methods on $L$-smooth problems, the iteration complexity $\mathcal{O}(L\Delta/\epsilon^2)$ is known to be optimal~\citep{carmon2019lowerbounds}.
A classical approach to achieving this rate without prior knowledge of the problem parameters is to use linesearch~\citep{armijo1966, beck2017}.
However, linesearch can be computationally expensive when function evaluations are costly 
and can make theoretical guarantees difficult to establish when function values are stochastic. 
This has motivated recent analyses of linesearch-free algorithms with an $\mathcal{O}(1/\epsilon^2)$ dependence on $\epsilon$~\citep{lan2024projectedgradientmethodsnonconvex, yagishita2025simplelsfree}.

When the objective function is convex, the iteration complexity for finding an $\epsilon$-optimal solution, i.e., a point $x\in\mathbb{R}^d$ satisfying $f(x)-f^*\leq\epsilon$, has also been extensively studied.
Linesearch methods are known to achieve an iteration complexity of $\mathcal{O}(LR^2/\epsilon)$~\citep{beck2017},
and linesearch-free algorithms with the same $\mathcal{O}(1/\epsilon)$ dependence on $\epsilon$ have also been studied~\citep{malitsky2020adaptivegradientdescentdescent, hoai2024novelstep, latafat2025adaptive, malitsky2024adaptive}.
On the other hand, the iteration-complexity lower bound for convex $L$-smooth problems is $\Omega(\sqrt{L/\epsilon}R)$~\citep{nesterov2004intro}. 
Recently, accelerated methods have been developed that attain this complexity while being both linesearch-free and parameter-free~\citep{li2025simplewolinesearch, suh2025adaptive, borodich2026nesterovgraal}.

\paragraph*{Methods for Deterministic $(L_0,L_1)$-Smooth Problems}

The notion of $(L_0,L_1)$-smoothness was introduced by \citet{zhang2020gradientclippingacceleratestraining}, who pointed out that Clip.GD and Norm.GD can efficiently find stationary points in nonconvex problems.
Here, we review prior work on deterministic $(L_0,L_1)$-smooth optimization.
Table~\ref{tbl-refs} summarizes the methods most closely related to our work and their iteration complexities in the convex and nonconvex settings.

We first summarize prior work in the convex setting and the corresponding iteration complexities for finding an $\epsilon$-optimal solution.
When the true values of $(L_0,L_1)$ are available, Clip.GD and Norm.GD are known to achieve an iteration complexity of $\mathcal{O}(L_0R^2/\epsilon + L_1 R^2)$~\citep{gorbunov2025methods, vankov2025optimizingl0l1smooth}.
Moreover, analogous to acceleration for smooth convex optimization, there exist methods for $(L_0,L_1)$-smooth problems 
that use the values of $(L_0,L_1)$ and achieve an iteration complexity with $\mathcal{O}(1/\sqrt{\epsilon})$ dependence on $\epsilon$ 
~\citep{li2023convex, gorbunov2025methods, tyurin2026nearoptimalconvergenceacceleratedgradient}.

We next review parameter-free methods in the convex setting.
\citet{takezawa2024paramfreeCGD} showed that an inexact Polyak stepsize using a lower bound $l^*$ on $f^*$ achieves an iteration complexity of $\mathcal{O}([L_0R^2+(l^*-f^*)^2]/\epsilon^2)$.
However, the resulting $(1/\epsilon^2)$ dependence is suboptimal, the additional error term $(l^*-f^*)^2$ further deteriorates the complexity bound, and the method requires prior knowledge of a valid lower bound $l^*$.
\citet{gorbunov2025methods} showed that AdGD, proposed by \citet{malitsky2020adaptivegradientdescentdescent}, is both parameter-free and linesearch-free and achieves an iteration complexity of $\mathcal{O}(L_0\mathcal{D}^2/\epsilon)$.
Here, $\mathcal{D}$ is a constant depending on $\|x^1-x^*\|^2$, $f(x^0)-f^*$, and the initial update.
\citet{vankov2025optimizingl0l1smooth, borodich2026nesterovgraal} further developed an  accelerated method that requires neither parameter tuning nor linesearch, with the latter additionally avoiding an inner loop.

While such tuning-free accelerated methods have been developed for the convex setting,
there has been surprisingly little work on parameter-free methods for deterministic nonconvex $(L_0,L_1)$-smooth problems.
\citet{hubler2024parameter} analyzed the iteration complexity of a linesearch method with monotonically increasing step sizes for $(L_0,L_1)$-smooth problems. 
However, its $L_1^2$ term has a $1/\epsilon^2$ dependence, which does not match the best-known complexity.
This dependence arises because the curvature upper bound $L_0+L_1\|\nabla f\|$ itself depends on the gradient, causing the method to retain the influence of gradients from early iterations.

As the above discussion shows, to the best of our knowledge, 
no prior work has established an iteration-complexity guarantee for a parameter-free method that simultaneously retains the leading-order complexity $\mathcal{O}(L_0R^2/\epsilon)$ for finding an $\epsilon$-optimal solution in the convex setting and matches, in the nonconvex setting, the complexity order of methods with known $(L_0,L_1)$.

\begin{table}[tb]
    \centering
    \small
    \renewcommand{\arraystretch}{1.6}
    \caption{Comparison of major related methods and their complexities in the convex and nonconvex settings.\newline
    \small
    $\tilde{\mathcal{O}}$ indicates the dependence on $L_0,L_1,R,\Delta$, and $\epsilon$.
    $\epsilon$-stat./$\epsilon$-opt. denote the iteration complexities for finding an $\epsilon$-stationary point and an $\epsilon$-optimal solution, respectively.
    NGM-RA is the only parameter-free method that provides these complexity guarantees in both the convex and nonconvex settings at the same time.  As \citet{vankov2025optimizingl0l1smooth} stated, note that  $\mathcal{O}\left( \frac{L_0 R^2}{\epsilon} + L_1 R \log \frac{\Delta}{\epsilon}\right) \subset \mathcal{O}\left( \frac{L_0 R^2}{\epsilon} + L_1^2 R^2\right)$.
    }\label{tbl-refs}
    \begin{tabular}{@{}L{0.22\linewidth} C{0.17\linewidth} C{0.22\linewidth} C{0.28\linewidth}@{}}
    \toprule
    Reference
      & parameter-\newline free
      & $\epsilon$-stat.\newline(non-convex)
      & $\epsilon$-opt.\newline(convex) \\
    \midrule
    Clip.GD/Norm.GD\newline \citet{zhang2020gradientclippingacceleratestraining}
    \newline \citet{vankov2025optimizingl0l1smooth}
        & require \(L_0,L_1\) &
        \(\tilde{\mathcal{O}}\left( \frac{L_0 \Delta}{\epsilon^2} + \frac{L_1 \Delta}{\epsilon} \right)\) &
        \(\tilde{\mathcal{O}}\left( \frac{L_0 R^2}{\epsilon} + L_1 R \log \frac{\Delta}{\epsilon}\right)\) \\
    Algorithm 2 in\newline
        \citet{hubler2024parameter}
        & \(\checkmark\) &
        \(\tilde{\mathcal{O}}\left(  \frac{L_0 \Delta + L_1^2 \Delta^2}{\epsilon^2}\right)\)
        & \(-\) \\
    AdGD\newline \citet{gorbunov2025methods}
        & \(\checkmark\) & \(-\) &
        \(\tilde{\mathcal{O}}\left( \frac{L_0 R^2}{\epsilon} + L_1^6 R^6 \right)\) \\
    Algorithm 2 in \newline \citet{tyurin2026nearoptimalconvergenceacceleratedgradient}
        & require \(L_0,L_1,R,\Delta\) & \(-\) &
        \(\tilde{\mathcal{O}}\left( \sqrt{\frac{L_0 R^2}{\epsilon}} + L_1 R \log \frac{\Delta}{\epsilon}\right)\) \\
    Accelerated GRAAL\newline\citet{borodich2026nesterovgraal}
        &\(\checkmark\) & \(-\) &
        \(\tilde{\mathcal{O}}\left( \sqrt{\frac{L_0 R^2}{\epsilon}} + L_1^3 R^3\right)\) \\
    \textbf{\hyperref[algo-ngm-ra]{NGM-RA} (ours)} 
        & \(\checkmark\) &
        \(\tilde{\mathcal{O}}\left( \frac{L_0 \Delta}{\epsilon^2} + \frac{L_1 \Delta}{\epsilon} \right)\) &
        \(\tilde{\mathcal{O}}\left( \frac{L_0 R^2}{\epsilon} + L_1 R \log \frac{\Delta}{\epsilon}\right)\) \\
    \bottomrule
    \end{tabular}
\end{table}

\subsection{Assumptions and Notation}\label{assumptions-and-notation}

Throughout this paper, we assume that the objective function $f$ satisfies the regularity conditions in Assumption~\ref{asm-basic} 
and the $(L_0,L_1)$-smoothness condition in Assumption~\ref{asm-l0l1sm}.

\begin{assumption}\label{asm-basic}
The function $f$ satisfies the following conditions.

\begin{enumerate}
\item[(a)] $f$ is bounded below:
$f^* \defeq \inf_{x\in\mathbb{R}^d} f(x)>-\infty$.
\item[(b)] $f$ is twice continuously differentiable.
\end{enumerate}
\end{assumption}

\begin{assumption}\label{asm-l0l1sm}
The function $f$ is $(L_0,L_1)$-smooth for some $(L_0,L_1)\in\R^2_{\geq 0}$.
\end{assumption}

When discussing convergence to the optimal value in the convex setting, we additionally impose Assumption~\ref{asm-convex}.

\begin{assumption}\label{asm-convex}
The function $f$ is convex, and there exists an optimal solution $x^*$ satisfying $f(x^*)=f^*$.
\end{assumption}

\section{Normalized Gradient Method with Region-based Adaptation}\label{gradient-method-with-parameter-localization-for-adaptive-normalization}

Intuitively, the normalized gradient method with region-based adaptation (\hyperref[algo-ngm-ra]{NGM-RA}) maintains the estimated region $E_k$ of $(L_0,L_1)$ values 
that satisfy Equation~\ref{eq-l0l1-def} in Definition~\ref{def-l0l1sm} over the iterates $\{x_l\}_{l=0}^k$, 
and performs updates analogous to the normalized gradient method using $(L_0,L_1)\in E_k$ chosen to maximize the step size.
NGM-RA can also be viewed as an analogue of cutting-plane and localization methods~\citep{boyd2007localization} for problem parameters, with Equation~\ref{eq-l0l1-def} serving as a valid inequality.
In this section, we present NGM-RA, summarized in Algorithm~\ref{algo-ngm-ra}, and provide several remarks on its details.
The initial region $O_{\xi,r_0}$ and the function $\Phi$ used in the algorithm are defined as follows:
\[
\begin{aligned}
O_{\xi, r_0} &\defeq \left\{ (L_0,L_1) \in \R^2 \relmid L_0 \geq \xi L_1, \, L_1 \geq \xi L_0, \,L_0 + L_1 \geq \sqrt{2} r_0\right\},
\\
\Phi(s) &\defeq 2 \frac{e^s - 1 - s}{s^2}\; \text{for} \; s>0 \; \text{and} \; \Phi(0) \defeq 1.
\end{aligned}
\]
Both $\Phi$ and the mapping $s \mapsto s\Phi(cs)$, for any constant $c>0$, are strictly increasing on $s\ge0$, as can be verified from their power-series expansions.

\floatstyle{ruled}
\restylefloat{algorithm}
\begin{algorithm}[t]
\caption{\label{algo-ngm-ra}NGM-RA}
\begin{algorithmic}[1]
\Require $x^0 \in \R^d, r_0 >0, 0<\xi <1, 2 \leq 2 \beta < \alpha, \, \beta < \frac{1}{2\xi}\, \Phi\left(\frac{1}{\alpha \xi}\right), 0< \tau \leq 1$
\State $E_0 \gets O_{\xi, r_0}, \, k\gets 0$
\Repeat
  \State $g_k \gets \|\nabla f(x^k)\|$
  \State $\gamma_k \gets \min_{(l_0,l_1) \in E_k} l_0 + l_1 g_k$
  \State $x'^{k+1} \gets x^{k} - \frac{1}{\alpha \gamma_k} \nabla f(x^k)$
  \State $M_k \gets \frac{2\alpha^2 \gamma_k^2}{g_k^2} \left[f(x'^{k+1}) - f(x^k) + g_k^2/(\alpha \gamma_k)\right]$
  \If{$M_k > 2\beta \gamma_k$}
    \State \textbf{Compute} $s_{k}$ \textbf{solving} $s_k \Phi\left( \frac{s_k}{\alpha} \right) = \frac{M_k}{\gamma_k}$
    \State $E_{k+1} \gets E_k \cap \left\{(l_0,l_1) \relmid l_0 + l_1 g_k \geq s_k \gamma_k\right\}$
    \State $x^{k+1} \gets x^k$
  \Else
    \State $E_{k+1} \gets O_{\xi, r_0} \cap \tau E_k$
    \State $x^{k+1} \gets x'^{k+1}$
  \EndIf
  \State $k \gets k+1$
\Until{Termination criterion is satisfied.}
\end{algorithmic}
\end{algorithm}
\floatstyle{plain}

The constants $\alpha,\beta$, and $\xi$ used in the algorithm are chosen in advance as follows.
First, we choose the threshold $\beta$ for determining whether a tentative update is accepted or rejected such that $\beta\geq 1$.
Next, we choose the coefficient $\alpha$, which makes the estimate conservative, such that $\alpha>2\beta$.
Finally, to ensure that the estimation procedure terminates after finitely many iterations, we choose $\xi>0$ sufficiently small so that
$\beta < \frac{1}{2\xi}\Phi\left(\frac{1}{\alpha\xi}\right)$.
The parameter $\tau$ can be chosen to satisfy $\tau\leq1$.
For example, one may take $\beta=1$, $\alpha=3$, $\xi=r_0=10^{-2}$, and $\tau=0.9$.
When $\tau<1$, the algorithm allows larger step sizes at successful iterations, whereas setting $\tau=1$ leads to improved iteration-complexity bounds; see Corollaries~\ref{cor-sta-total} and~\ref{cor-opt-total}.
In Section~\ref{sec-theory}, we establish iteration-complexity guarantees for $(L_0,L_1)\in O_{\xi,r_0}$.
Thus, choosing smaller $\xi$ and $r_0$ allows the method to adapt over a broader range of $(L_0,L_1)$ values and potentially exploit smaller parameter values. 
However, as shown in Subsection~\ref{sec-num-fail}, this comes with a trade-off: the number of unsuccessful iterations $N_{\mathrm{fail}}$,
during which the iterate is not updated, increases.
In practice, this effect is small, so we choose $\xi$ and $r_0$ sufficiently small as above.

In Algorithm~\ref{algo-ngm-ra}, $\gamma_k$ represents the curvature estimated from past iterates, 
while $M_k$ represents the newly observed curvature along the descent direction, since 
\[
M_k = 2\frac{f(x'^{k+1})-f(x^k) - \nabla f(x^k)^T (x'^{k+1}-x^k)}{\|x'^{k+1}-x^k\|^2}.
\]
At iteration $k$, we declare the curvature estimate unsuccessful if $M_k>2\beta\gamma_k$, and successful if $M_k\leq2\beta\gamma_k$.
In other words, the curvature estimate is considered successful 
when the estimated curvature is conservative relative to the observed curvature.
The region $E_k$ has the geometry illustrated in Figure~\ref{fig:ek-int}. 
Its update can be understood intuitively as follows.
An unsuccessful estimate indicates that the current region $E_k$ is too permissive and has led to an underestimate of the curvature.
We therefore restrict the $(L_0,L_1)$ region so that the condition corresponding to Equation~\ref{eq-l0l1-def} in Definition~\ref{def-l0l1sm} is satisfied for $x^{k+1}$ and $x^k$ (Figure~\ref{fig:ek-failed}).
For a successful estimate, reflecting the fact that the curvature estimate was conservative, 
we contract $E_k$ toward the origin by a factor of $\tau\,(\leq 1)$, thereby allowing more aggressive step sizes (Figure~\ref{fig:ek-success}).

\begin{figure}[t]
\centering

\begin{minipage}[b]{0.32\textwidth}
  \centering
  \includegraphics[width=\textwidth]{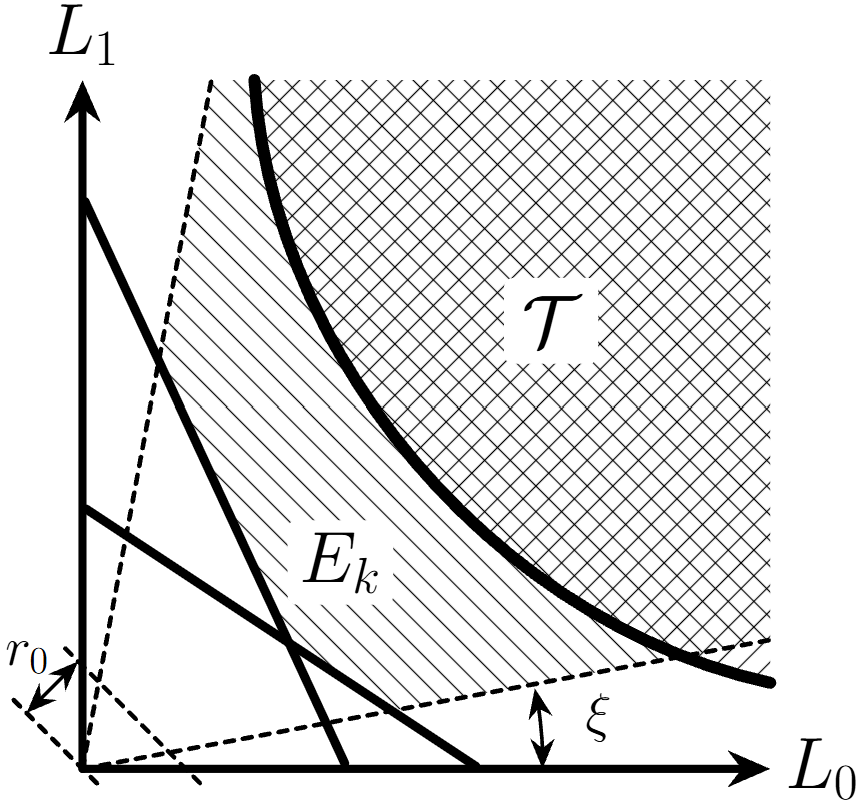}
  \caption{Geometry of the region $E_k$ and the true parameter region $\mathcal{T}$}
  \label{fig:ek-int}
\end{minipage}
\hfill
\begin{minipage}[b]{0.32\textwidth}
  \centering
  \includegraphics[width=0.95\textwidth]{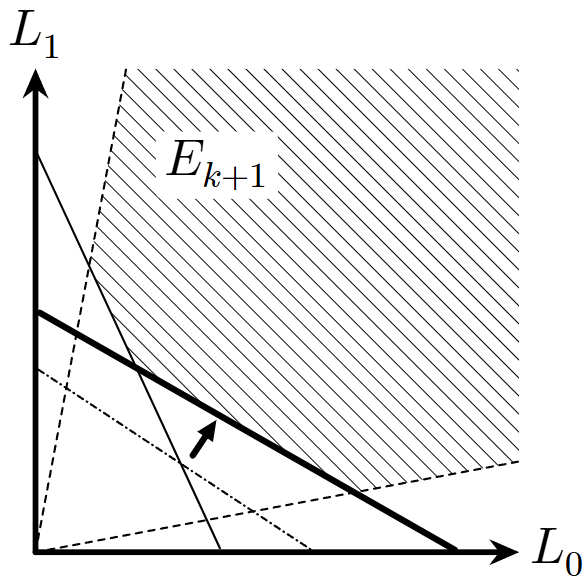}
  \caption{Cutting $E_k$ after an unsuccessful curvature estimate}
  \label{fig:ek-failed}
\end{minipage}
\begin{minipage}[b]{0.32\textwidth}
  \centering
  \includegraphics[width=0.95\textwidth]{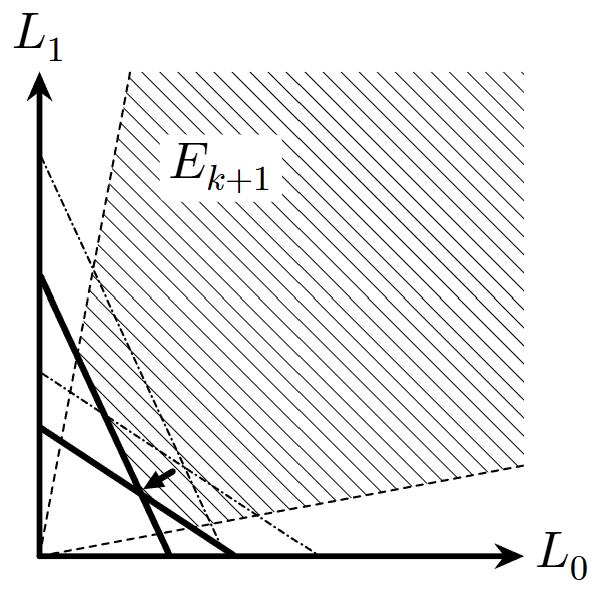}
  \caption{Expansion of $E_k$ after a successful curvature estimate}
  \label{fig:ek-success}
\end{minipage}

\end{figure}

As seen from the update rule for $E_k$ after a successful estimate, Algorithm~\ref{algo-ngm-ra} moves the boundary of $E_k$ toward the origin by a factor of $\tau<1$, thereby allowing the estimated region to expand.
Consequently, even if an excessively large curvature is observed near the initial point, its influence can decay geometrically over subsequent iterations.

The left-hand side of the equation
$s \Phi\left(\frac{s}{\alpha}\right)=\frac{M_k}{\gamma_k}$
used to determine $s_k$ is monotone in the scalar variable $s$. 
Hence, $s_k$ can be computed efficiently, for example, by binary search, and the associated computational cost is negligible.
The computation $\gamma_k \gets \min_{(l_0,l_1)\in E_k} l_0+l_1g_k$ 
is a two-dimensional linear program because $E_k$ is a convex polygon. Since it suffices to examine its vertices, the computational cost is negligible in practice.
Moreover, the number of vertices of $E_k$ is at most on the order of the number of unsuccessful estimates.

\section{Theoretical Analysis}\label{sec-theory}
In this section, we derive upper bounds on the iteration complexity of Algorithm~\ref{algo-ngm-ra} for finding an $\epsilon$-stationary point in the nonconvex setting and an $\epsilon$-optimal solution in the convex setting.

The main idea of the analysis is simple.
The acceptance criterion in Algorithm~\ref{algo-ngm-ra} is based on essentially the same inequality that plays a central role in the convergence analysis of algorithms with known $L_0$ and $L_1$.
For rejected updates, the estimated region shrinks, which allows us to bound the total number of rejected updates (Theorem~\ref{thm-failnum}).
For accepted updates, the existing convergence analysis can be applied with only minor modifications (Theorems~\ref{thm-eps-comp},~\ref{thm-conv-optv}).
This simplicity is an advantage of estimating a region of plausible values for the unknown parameters.

\subsection{Preliminaries}\label{sec-prep}
In this subsection, we introduce the notation used in the proofs and establish basic properties of the estimated region.
Following the notation used in Algorithm~\ref{algo-ngm-ra}, we denote the gradient norm at iteration $k$ by $g_k$, the estimated curvature along the gradient direction by $\gamma_k$, and the observed curvature by $M_k$:
\[
\begin{aligned}
g_k
&\defeq \|\nabla f(x^k)\|,\\
\gamma_k
&\defeq \min_{(l_0,l_1)\in E_k} l_0+l_1g_k,\\
M_k
&\defeq \frac{2\alpha^2 \gamma_k^2}{g_k^2} \left[f(x'^{k+1}) - f(x^k) + g_k^2/(\alpha \gamma_k)\right].
\end{aligned}
\]

We further introduce notation used in the proofs.
We define the set of $(L_0,L_1)$-smoothness parameters for the objective function $f$ as
\[
\mathcal{T} \defeq \left\{
(L_0,L_1)\in\R^2_{\geq 0}
\relmid
\forall x\in\mathbb{R}^d,
\|\nabla^2 f(x)\| \leq L_0+L_1\|\nabla f(x)\|
\right\}.
\]

Let
$\tilde{\mathcal{T}}\defeq\mathcal{T}\cap O_{\xi,r_0}$,
which is nonempty, and fix an arbitrary $(L_0,L_1)\in\tilde{\mathcal{T}}$.
Note that these values need not be known by the algorithm. Moreover, all subsequent upper bounds hold for every $(L_0,L_1)\in\tilde{\mathcal{T}}$ and can therefore be optimized by taking the infimum over $\tilde{\mathcal{T}}$.
We also define the following cone:
\[
C_\xi \defeq \left\{(l_0,l_1)  \in \R^2_{\geq 0}\relmid l_0 \geq \xi l_1, \, l_1 \geq \xi l_0\right\}.
\] 
Using this notation, the initial estimated region can be expressed as
$O_{\xi,r_0}=C_\xi\cap\left\{(l_0,l_1) \in\mathbb{R}_{\geq0}^2 \relmid l_0+l_1\geq\sqrt{2}r_0\right\}$.

We use the following upper bound on the directional curvature in the subsequent analysis.

\begin{lemma}[Upper bound on curvature~\citep{vankov2025optimizingl0l1smooth}, Lemma
2.5]\label{lem-upper}

Suppose that Assumptions~\ref{asm-basic} and~\ref{asm-l0l1sm} hold. Define the directional curvature at $x$ along $d$ by
\[
M(x,d) \defeq 2\frac{f(x+d) - f(x) - \nabla f(x)^T d}{\|d\|^2}.
\] 
Then, for any $(L_0,L_1)\in\mathcal{T}$,
\[
M(x,d) \leq (L_0 + L_1 \|\nabla f(x)\|) \Phi(L_1 \|d\|).
\]

\end{lemma}

For the $L$-smooth setting, \citet{lan2024projectedgradientmethodsnonconvex} and \citet{yagishita2025simplelsfree} use the quantity $\max_{i<k}M_i$ as an estimate of the smoothness parameter $L$ for step-size selection.
We see that $M_i\le L$, and hence each $M_i$ is a lower bound on the true smoothness parameter $L$ observed in each iteration.
Since $\max_{i<k} M_i = \min \left\{l \ge 0 \relmid l \ge M_i,\, \forall i < k\right\}$,
these methods can be interpreted as progressively narrowing down the set of plausible values of the true parameter and then selecting the smallest admissible value, up to a conservative factor, to obtain an aggressive step size. 
This interpretation suggests a natural extension to multiple coupled unknown parameters: instead of maintaining a scalar estimate, we maintain a region of parameter values consistent with the observations and select an aggressive parameter pair from that region for step-size selection. 
From this perspective, Lemma~\ref{lem-upper} suggests updating the estimated region as follows.
\[
E_{k+1} \gets E_k \cap \left\{(l_0,l_1) \in\R^2_{\geq 0} \relmid (l_0 + l_1 g_k) \Phi\left(\frac{l_1 g_k}{\alpha \gamma_k}\right) \geq M_k \right\}.
\]
However, this update would make $E_k$ a nonlinear region and complicate the computation of $\gamma_k$.
To avoid this difficulty, we choose $s$ such that
\[
\left\{(l_0,l_1) \in\R^2_{\geq 0} \relmid (l_0 + l_1 g_k) \Phi\left(\frac{l_1 g_k}{\alpha \gamma_k}\right) \geq M_k \right\} \subset \left\{(l_0,l_1) \in\R^2_{\geq 0} \relmid l_0 + l_1 g_k\geq s \gamma_k\right\}
\] 
and update the estimated region using the resulting linear boundary.
The following lemma guarantees this inclusion. The proof can be found in
Appendix~\ref{sec-prep-proof}.

\begin{tcolorbox}[mytheoremstyle]
\begin{restatable}[$\tilde{\mathcal{T}}\subset E_k$]{lemma}{lemKkconsv}\label{lem-Kk-consv}

Suppose that Assumptions~\ref{asm-basic} and~\ref{asm-l0l1sm} hold.
Then, $\tilde{\mathcal{T}}\subset E_k$ for all $k$.
\end{restatable}
\end{tcolorbox}

The success condition can be rewritten in the following form, which is key to establishing the convergence guarantees.

\begin{tcolorbox}[mytheoremstyle]
\begin{restatable}[Descent inequality]{lemma}{lemdescent}\label{lem-descent}
Suppose that Assumptions~\ref{asm-basic} and~\ref{asm-l0l1sm} hold, and let $(L_0,L_1)\in\tilde{\mathcal{T}}$. At every successful iteration of Algorithm~\ref{algo-ngm-ra}, the following inequality holds.
\[
f(x^{k+1}) - f(x^k) \leq - \frac{\alpha - \beta}{\alpha^2\gamma_k} \|\nabla f(x^k)\|^2 \leq - \frac{\alpha - \beta}{\alpha^2} \frac{g_k^2}{L_0 + L_1 g_k}.
\]
\end{restatable}
\end{tcolorbox}

\begin{proof}[Proof sketch]
The result follows directly by rearranging the acceptance condition:
\[
M_k = 2\frac{f(x^{k+1}) - f(x^k) + \frac{1}{\alpha \gamma_k} \|\nabla f(x^k)\|^2 }{\|x^{k+1} - x^k\|^2}\leq 2\beta \gamma_k.
\]
\end{proof}

\subsection{Upper Bound on the Number of Failures}\label{sec-num-fail}

The total number of iterations is the sum of successful and unsuccessful iterations. 
In this subsection, let $N_{\mathrm{succ}}^{(k)}$ and $N_{\mathrm{fail}}^{(k)}$ denote the numbers of successful and unsuccessful iterations up to iteration $k$, respectively. 
We derive an upper bound on $N_{\mathrm{fail}}^{(k)}$ using $N_{\mathrm{succ}}^{(k)}$.
This analysis applies to both the convex and nonconvex settings.

To avoid overestimating the size of regions far from the origin, we introduce the following measure.
\[
A\in \mathcal{M}(\R^2), \, \nu(A) \defeq \int_A \frac{\mathrm{d} l_0 \mathrm{d} l_1}{l_0^2 + l_1^2}.
\] 
Here, $\mathcal{M}(\mathbb{R}^2)$ denotes the collection of Lebesgue-measurable subsets of $\mathbb{R}^2$.
This measure is scale-invariant: $\nu(A)=\nu(tA)$ for any $t>0$.
Let $\mathcal{E}\defeq O_{\xi,r_0}\backslash\tilde{\mathcal{T}}$ denote the region ,
and let $F_k\defeq O_{\xi,r_0}\backslash
E_k$ denote the region excluded from the estimated region by unsuccessful iterations. 
By Lemma~\ref{lem-Kk-consv}, we have $F_k\subset\mathcal{E}$ for all $k$.

Furthermore, let $S_I$ be a lower bound on the increase of $F_k$ in the measure at a unsuccessful iteration, 
and let $S_D$ be an upper bound on its decrease at a successful iteration. 
By the additivity of the measure, $N_\mathrm{fail}^{(k)}$ can then be bounded as in Equation~\ref{eq-nf-funda}.
\begin{equation}\protect\phantomsection\label{eq-nf-funda}{
\nu(\mathcal{E}) \geq \nu(F_k)\geq N_\mathrm{fail}^{(k)} S_I - N_{\mathrm{succ}}^{(k)} S_D \quad
\mathrm{i.e.}
\quad N_\mathrm{fail}^{(k)} \leq \frac{\nu(\mathcal{E})}{S_I} + N_\mathrm{succ}^{(k)} \frac{S_D}{S_I}.
}\end{equation}

We therefore proceed by evaluating $S_I$, $S_D$, and $\nu(\mathcal{E})$.
The detailed proofs are provided in Appendix~\ref{sec-fail-proof}.

\begin{tcolorbox}[mytheoremstyle]
\begin{restatable}[Lower bound $S_I$ on the increase in $\nu(F_k)$ at a unsuccessful iteration]{lemma}{lemdkarea}\label{lem-dk-area}

At a unsuccessful iteration, the increase in $\nu(F_k)$ is bounded below as follows. 
\[
\nu(F_{k+1}) -  \nu(F_{k})  \geq \frac{\xi (\mu-1)^2}{2 (1+\xi^2)\mu^2} \fedeq S_I.
\]
Here, $\mu>1$ is the positive constant satisfying $\mu\Phi(\mu/\alpha)=2\beta$.

\end{restatable}
\end{tcolorbox}

\begin{tcolorbox}[mytheoremstyle]
\begin{restatable}[Upper bound $S_D$ on the decrease in $\nu(F_k)$ at a successful iteration]{lemma}{lemseval}\label{lem-se-val}

At a successful iteration, the decrease in $\nu(F_k)$ is bounded above as follows.
\[
\nu(F_k) -  \nu(F_{k+1}) \leq \frac{\pi}{2} \log \frac{1}{\tau} \fedeq S_D.
\] 

\end{restatable}
\end{tcolorbox}

We can further bound the measure of $\mathcal{E}$ as follows.

\begin{tcolorbox}[mytheoremstyle]
\begin{restatable}[Upper bound on the measure of $\mathcal{E}$]{lemma}{lemEarea}\label{lem-E-area}

Suppose that Assumptions~\ref{asm-basic} and~\ref{asm-l0l1sm} hold, and let $(L_0,L_1)\in\tilde{\mathcal{T}}$. Then, the following bound holds.
\[
\nu(\mathcal{E}) \leq \frac{\pi}{2} \log \left(\frac{1}{r_0} \sqrt{\left(\frac{1}{\xi^2} + 1\right)(L_0^2 + L_1^2)}\right) .
\]

\end{restatable}
\end{tcolorbox}

Combining the above bounds, we obtain the following upper bound on $N_\mathrm{fail}$.

\begin{tcolorbox}[mytheoremstyle]
\begin{restatable}[Upper bound on the number of unsuccessful iterations]{theorem}{thmfailnum}\label{thm-failnum}

Suppose that Assumptions~\ref{asm-basic} and~\ref{asm-l0l1sm} hold, and let $(L_0,L_1)\in\tilde{\mathcal{T}}$.
For Algorithm~\ref{algo-ngm-ra}, $N_\mathrm{fail}^{(k)}$, which is the number of unsuccessful iterations up to iteration $k$, is bounded as follows.
\[
\begin{aligned}
N_\mathrm{fail}^{(k)}
&\leq \frac{\pi(1+\xi^2)\mu^2}{\xi (\mu-1)^2}\left[\log \left(\frac{1}{r_0} \sqrt{\left(\frac{1}{\xi^2} + 1\right)(L_0^2 + L_1^2)}\right) 
+ N_{\mathrm{succ}}^{(k)} \log \frac{1}{\tau}\right]
\\
&\in \mathcal{O}\left( \frac{1}{\xi} \left[
 \log \left(\frac{\sqrt{L_0^2 + L_1^2}}{r_0 \xi}\right) + N_{\mathrm{succ}}^{(k)} \log \frac{1}{\tau}\right]
\right).
\end{aligned}
\]

\end{restatable}
\end{tcolorbox}

\begin{proof}
By Lemmas~\ref{lem-dk-area}, \ref{lem-se-val}, and \ref{lem-E-area}, together with Equation~\ref{eq-nf-funda}, we obtain the result.
\end{proof}

Setting $\tau=1$ gives an upper bound independent of $N_\mathrm{succ}^{(k)}$.

\subsection{Iteration Complexity for Convex Problems}\label{sec-cvx}
In this subsection, under Assumptions~\ref{asm-basic}, \ref{asm-l0l1sm}, and \ref{asm-convex}, we derive an upper bound on the iteration complexity of Algorithm~\ref{algo-ngm-ra} for convex $(L_0,L_1)$-smooth problems.
The detailed proof is provided in Appendix~\ref{sec-cvx-proof}.

\begin{tcolorbox}[mytheoremstyle]
\begin{restatable}[Number of successful iterations to an $\epsilon$-optimal solution]{theorem}{thmconvoptv}\label{thm-conv-optv}
Suppose that Assumptions~\ref{asm-basic}, \ref{asm-l0l1sm}, and \ref{asm-convex} hold, and let $(L_0,L_1)\in\tilde{\mathcal{T}}$.
The number of successful iterations $N_{\epsilon,\mathrm{opt}}$ required by Algorithm~\ref{algo-ngm-ra} to find an $\epsilon$-optimal solution is bounded as follows.
\[
\begin{aligned}
N_{\epsilon, \mathrm{opt}} 
&\leq 1+\frac{\alpha^2}{\alpha - \beta}\left(L_0 R^2 \frac{1}{\epsilon} + L_1 R \log \frac{\Delta}{\epsilon}\right)\\
&\in \mathcal{O}\left( \frac{L_0 R^2}{\epsilon} + L_1 R \log \frac{\Delta}{\epsilon} \right) \subset \mathcal{O}\left( \frac{L_0 R^2}{\epsilon} + (L_1 R)^2 \right).
\end{aligned}
\]
\end{restatable}
\end{tcolorbox}

The proof follows essentially the same argument as Theorem 3.2 of \citet{vankov2025optimizingl0l1smooth} for an algorithm with known parameters.
This ability to leverage existing parameter-dependent analyses is one major advantage of our methodology.

Combining Theorem~\ref{thm-failnum} with the above result yields the following bound on the total iteration complexity.

\begin{corollary}[Total number of iterations to an $\epsilon$-optimal solution for convex $(L_0,L_1)$-smooth functions]\protect\hypertarget{cor-opt-total}{}\label{cor-opt-total}
Suppose that Assumptions~\ref{asm-basic}, \ref{asm-l0l1sm}, and \ref{asm-convex} hold, and let $(L_0,L_1)\in\tilde{\mathcal{T}}$.
The total number of iterations $N_\mathrm{fail} + N_{\epsilon,\mathrm{opt}}$ required by Algorithm~\ref{algo-ngm-ra} to find an $\epsilon$-optimal solution is bounded as follows.

\[
\begin{aligned}
&N_\mathrm{fail} + N_{\epsilon,\mathrm{opt}} \\
&\in \mathcal{O}\left( \frac{1}{\xi} \left[
 \log \left(\frac{\sqrt{L_0^2 + L_1^2}}{r_0 \xi}\right) + \left(\frac{L_0 R^2}{\epsilon} + L_1 R \log \frac{\Delta}{\epsilon}\right) \left(\log \frac{1}{\tau}+\xi\right)\right]
\right).
\end{aligned}
\]
\end{corollary}

In particular, when $\tau = 1$, the bound simplifies to
\[
N_\mathrm{fail} + N_{\epsilon,\mathrm{opt}} \in \mathcal{O}\left( \frac{1}{\xi}
 \log \left(\frac{\sqrt{L_0^2 + L_1^2}}{r_0 \xi}\right) + \left(\frac{L_0 R^2}{\epsilon} + L_1 R \log \frac{\Delta}{\epsilon}\right)\right),
\]
which eliminates the $1/\xi$ factor from the $\epsilon$-dependent term.

\subsection{Iteration Complexity for Nonconvex Problems}\label{sec-ncvx}

In this subsection, we show that, in the nonconvex setting, the number of successful iterations required to find an $\epsilon$-stationary point is $\mathcal{O}\left(\left(\frac{L_0}{\epsilon^2}+\frac{L_1}{\epsilon}\right)\Delta\right)$.
Since Lemma~\ref{lem-descent} already provides the required descent inequality at every successful iteration, the result follows from a straightforward telescoping-sum argument.
The detailed proof is provided in Appendix~\ref{sec-ncvx-proof}.

\begin{tcolorbox}[mytheoremstyle]
\begin{restatable}[Number of successful iterations to an $\epsilon$-stationary point]{theorem}{thmepscomp}\label{thm-eps-comp}
Suppose that Assumptions~\ref{asm-basic} and~\ref{asm-l0l1sm} hold, and let $(L_0,L_1)\in\tilde{\mathcal{T}}$.
The number of successful iterations $N_{\epsilon,\mathrm{sta}}$ required by Algorithm~\ref{algo-ngm-ra} to find an $\epsilon$-stationary point is bounded as follows.
\[
\begin{aligned}
N_{\epsilon,\mathrm{sta}} 
&\leq\frac{\alpha^2}{\alpha - \beta} \left(\frac{L_0}{\epsilon^2} + \frac{L_1}{\epsilon}\right)  \Delta
\in \mathcal{O}\left( \left(\frac{L_0}{\epsilon^2} + \frac{L_1}{\epsilon}\right) \Delta \right).
\end{aligned}
\]

\end{restatable}
\end{tcolorbox}

\begin{proof}[Proof sketch]
Summing the inequality in Lemma~\ref{lem-descent} over $k$ and using the fact that $x\mapsto x^2/(L_0+L_1x)$ is nondecreasing for $x\geq0$ yields the result.
\end{proof}

Combining Theorem~\ref{thm-failnum} with the above result yields the following bound on the total iteration complexity.

\begin{corollary}[Upper bound on the total number of iterations to an $\epsilon$-stationary point]\protect\hypertarget{cor-sta-total}{}\label{cor-sta-total}
Suppose that Assumptions~\ref{asm-basic} and~\ref{asm-l0l1sm} hold, and let $(L_0,L_1)\in\tilde{\mathcal{T}}$.
The total number of iterations $N_\mathrm{fail} + N_{\epsilon,\mathrm{sta}}$ required by Algorithm~\ref{algo-ngm-ra} to find an $\epsilon$-stationary point is of the following order.

\[
\begin{aligned}
&N_\mathrm{fail} + N_{\epsilon,\mathrm{sta}} \\
&\in \mathcal{O}\left( \frac{1}{\xi} \left[
 \log \left(\frac{\sqrt{L_0^2 + L_1^2}}{r_0 \xi}\right) + \left(\frac{L_0}{\epsilon^2} + \frac{L_1}{\epsilon}\right)\Delta \left(\log \frac{1}{\tau}+\xi\right)\right] \right).
\end{aligned}
\]
\end{corollary}

As in the nonconvex setting, when $\tau = 1$, the bound becomes
\[
N_\mathrm{fail} + N_{\epsilon,\mathrm{sta}} \in \mathcal{O}\left( \frac{1}{\xi}
 \log \left(\frac{\sqrt{L_0^2 + L_1^2}}{r_0 \xi}\right) + \left(\frac{L_0}{\epsilon^2} + \frac{L_1}{\epsilon}\right) \Delta \right).
\]

\section{Numerical Experiments}\label{sec-experiments}

Through numerical experiments, we demonstrate the following properties of \hyperref[algo-ngm-ra]{NGM-RA}.
\begin{enumerate}
\def\labelenumi{(\roman{enumi})}

\item It performs comparably to methods with the same order of iteration complexity in convex optimization
(Subsection~\ref{sec-poisson}).
\item It adapts to the true $L_0,L_1$ and selects aggressive step sizes
(Subsection~\ref{sec-poisson}).
\item It finds stationary points efficiently and stably in nonconvex optimization
(Subsection~\ref{sec-phase}).

\end{enumerate}

\subsection{Poisson Regression}\label{sec-poisson}

In this subsection, we consider the following Poisson regression as a convex smooth optimization problem and show that \hyperref[algo-ngm-ra]{NGM-RA} performs comparably to algorithms specifically designed for convex optimization.

\begin{equation}\label{eq-poisson}
\min_{x \in \R^d,\, c\in \R} \frac{1}{K} \sum_{k=1}^K \left[e^{c + a_k^T x} - b_k (c + a_k^T x)\right] + \frac{\lambda}{2} \|x\|^2, \quad b_k \in \mathbb{Z}_{\geq0}, \, a_k \in \R^d.
\end{equation}

By Proposition~\ref{prp-poisson}, letting $\overline{b} \defeq \frac{1}{K} \sum_k b_k, \, A \defeq \max_k 1+\|a_k\|^2$, the above objective function is $(A \overline{b}+\lambda,A)$-smooth.

We evaluate the empirical behavior of \hyperref[algo-ngm-ra]{NGM-RA} on the hourly subset of the UCI Bike Sharing Dataset~\citep{fanaee2014},  which has $m = 17379$ hourly observations and $n = 49$ features.
We compare it with the following four methods.

\begin{itemize}

\item
  AdGD : Adaptive Gradient Descent
 ~\citep{malitsky2020adaptivegradientdescentdescent}
\item
  aGRAAL : Accelerated Golden Ratio Algorithm
 ~\citep{borodich2026nesterovgraal}
\item
  Clip.GD : Clip.GD with $(L_0,L_1)=(A^2 \overline{b}+\lambda,A^2)$
~\citep{zhang2020gradientclippingacceleratestraining}
\item
  GD+Armijo : Gradient Descent with the monotone linesearch scheme analyzed by \citet{hubler2024parameter}
\end{itemize}

Figures~\ref{fig:bikehour} compare the methods on the Bike Sharing dataset.
We perform 20 runs with random initializations and report their geometric mean.

\begin{figure}[htbp]
\centering
\includegraphics[width=0.6\textwidth]{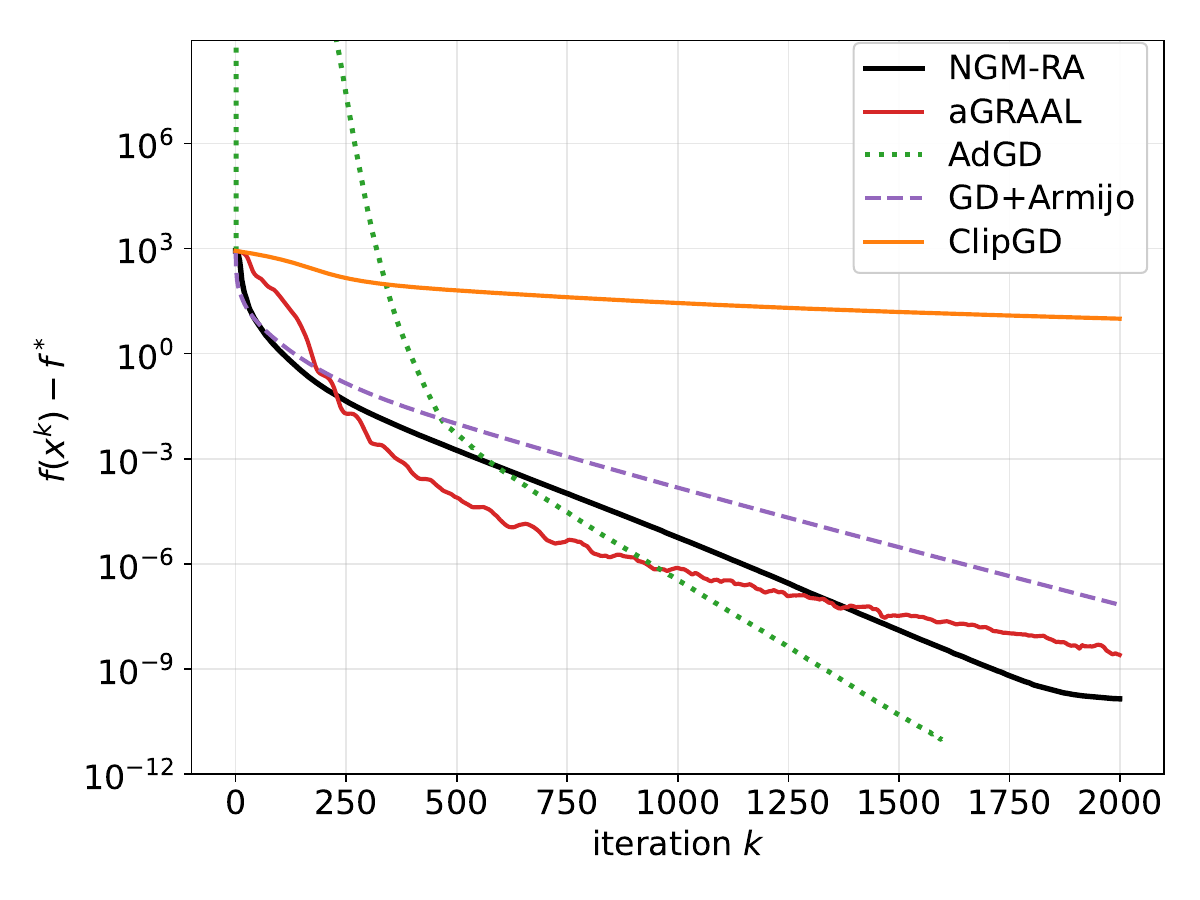}
\caption{Convergence of Each Method for Poisson Regression on BikeHour}
\label{fig:bikehour}
\end{figure}

The poor performance of Clip.GD is mainly due to the overly conservative bound of $L_0, L_1$. In contrast, estimating $(L_0,L_1)$ along the optimization trajectory allows \hyperref[algo-ngm-ra]{NGM-RA} to adapt to tighter values, which is a key advantage of our approach.
NGM-RA converges faster than the GD+Armijo method analyzed by \citet{hubler2024parameter}. While GD+Armijo monotonically increases its step size, NGM-RA maintains a two-dimensional estimated region for $(L_0,L_1)$ and can expand this region after successful iterations.
NGM-RA has a better iteration-complexity order than the bound established for AdGD by \citet{gorbunov2025methods}.
Nevertheless, AdGD exhibits faster convergence than NGM-RA in these experiments, while showing some instability during the early iterations.
As shown in Table~\ref{tbl-refs}, aGRAAL achieves the best dependence on $\epsilon$ among the considered methods for convex $(L_0,L_1)$-smooth problems in the small-$\epsilon$ regime. 
However, its empirical convergence is comparable to that of NGM-RA. 
One possible explanation is that the available iteration-complexity bound for aGRAAL has cubic dependence on $L_1$ and $R$, whereas that of NGM-RA has at most quadratic dependence on these quantities.
The observed performance difference 
therefore suggests that constant factors and the finite-accuracy regime play an important role in practice.

\subsection{Phase Retrieval}\label{sec-phase}

In this subsection, we consider phase retrieval~\citep{fannjiang2020numerics} as a nonconvex $(L_0,L_1)$-smooth optimization problem and demonstrate that \hyperref[algo-ngm-ra]{NGM-RA} can efficiently and stably find stationary points even in the nonconvex setting.
Phase retrieval aims to recover an unknown signal $z\in\mathbb{C}^n$ 
from intensity measurements $b_k=|a_k^\dagger z|^2$ associated with linear measurements $\{a_k\}_{k=1}^K\subset\mathbb{C}^n$.
Here, $a^\dagger$ denotes the conjugate transpose of $a\in\mathbb{C}^n$.
A wide range of problems in science and engineering can be formulated as phase retrieval, 
including X-ray crystallography~\citep{millane1990xray} and signal processing~\citep{ranieri2013phase}.

One common formulation of phase retrieval is the following least-squares problem.
\[
\min_{x \in \C^{n}} \sum_{k=1}^K \left[|a_{k}^\dagger x|^2   - b_k\right]^2.
\] 
To leverage the existing result that real phase retrieval is $(L_0,L_1)$-smooth, we restrict our attention to the real-valued setting, with $x\in\mathbb{R}^n$ and $a_k\in\mathbb{R}^n$, and consider the following real phase retrieval problem.

\begin{equation}\protect\phantomsection\label{eq-real-pr}{
\min_{x \in \R^{n}} \sum_{k=1}^K \left[(a_k^T x)^2 - b_k\right]^2.
}\end{equation}

As pointed out by \citet{bodard2026escaping}, when $\mathrm{span}\{a_k\}_k=\mathbb{R}^n$, the objective in Equation~\ref{eq-real-pr} is $(L_0,L_1)$-smooth but not $L$-smooth, since it is a quartic polynomial.
Furthermore, \citet{bolte2018first} showed that this objective satisfies relative smoothness with respect to the kernel $h(x)=\frac{1}{2}\|x\|^2+\frac{1}{4}\|x\|^4$.

In this subsection, we compare \hyperref[algo-ngm-ra]{NGM-RA} with the following four first-order methods that have convergence guarantees for Equation~\ref{eq-real-pr}.

\begin{itemize}

\item
  GD+Armijo (monotone) : Gradient Descent with the monotone Armijo line search analyzed by \citet{hubler2024parameter}
\item
  GD+Armijo (reset, $\eta_0=1$) : Gradient Descent that restarts the Armijo linesearch from the initial step size $\eta_0=1$ at every iteration
\item
  GD+Armijo (reset, $\eta_0=10$): Gradient Descent that restarts the Armijo linesearch from the initial step size $\eta_0=10$ at every iteration
\item
  AC-BPGM: the Auto-Conditioned Bregman Proximal Gradient Method of \citet{yagishita2025simplelsfree}, with kernel function $h(x)=\frac{1}{2}\|x\|^2+\frac{1}{4}\|x\|^4$ for the Bregman divergence
\end{itemize}

Figures~\ref{fig:phase-ret-iter} and~\ref{fig:phase-ret-orac} show the convergence of the gradient norm for each method as a function of the number of iterations and the number of calls to $f$ and $\nabla f$, respectively.
As in Subsection~\ref{sec-poisson}, we perform 20 runs with random initializations and report their geometric mean.

\begin{figure}[htbp]
  \centering
  \begin{minipage}[b]{0.48\textwidth}
    \centering
    \includegraphics[width=\textwidth]{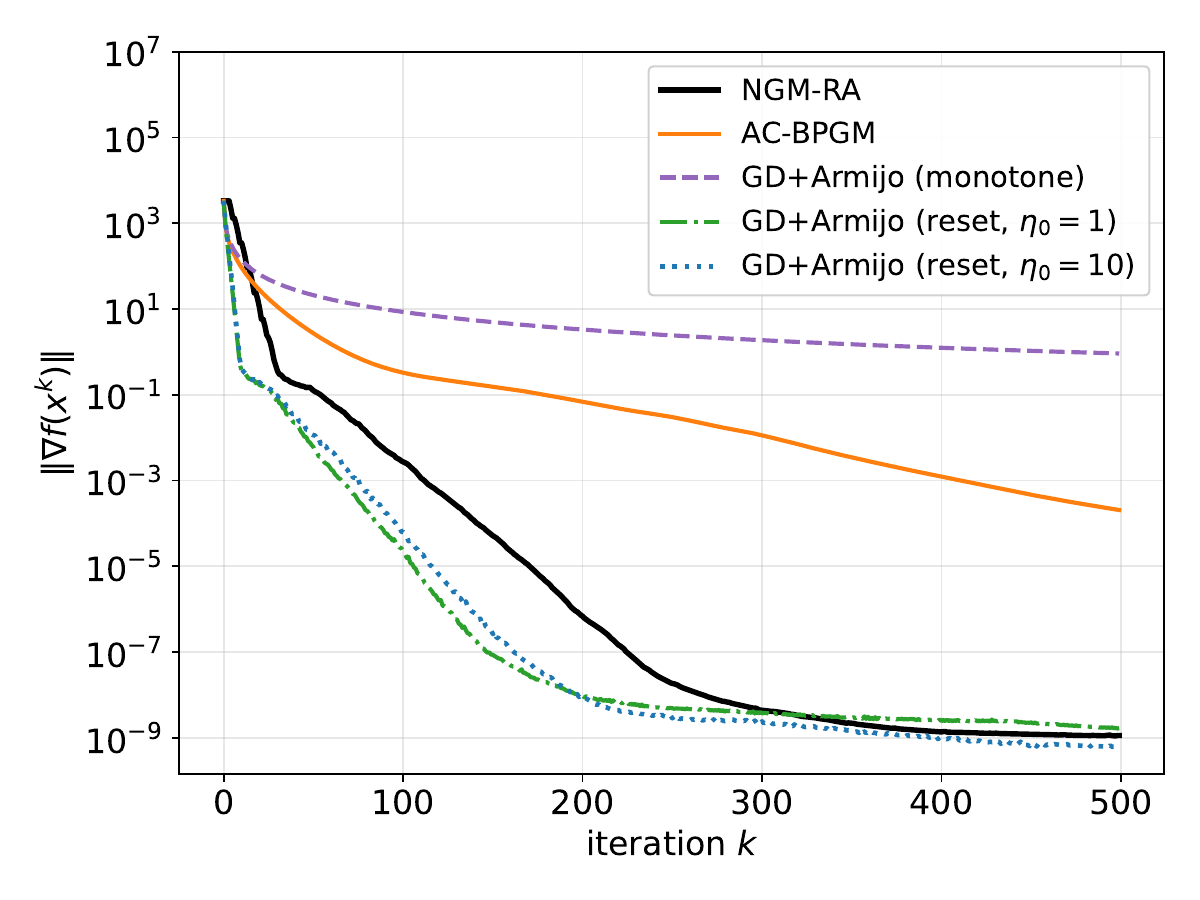}
    \caption{Convergence of Each Method versus the Number of Iterations for Real Phase Retrieval}
    \label{fig:phase-ret-iter}
  \end{minipage}
  \hfill
  \begin{minipage}[b]{0.48\textwidth}
    \centering
    \includegraphics[width=\textwidth]{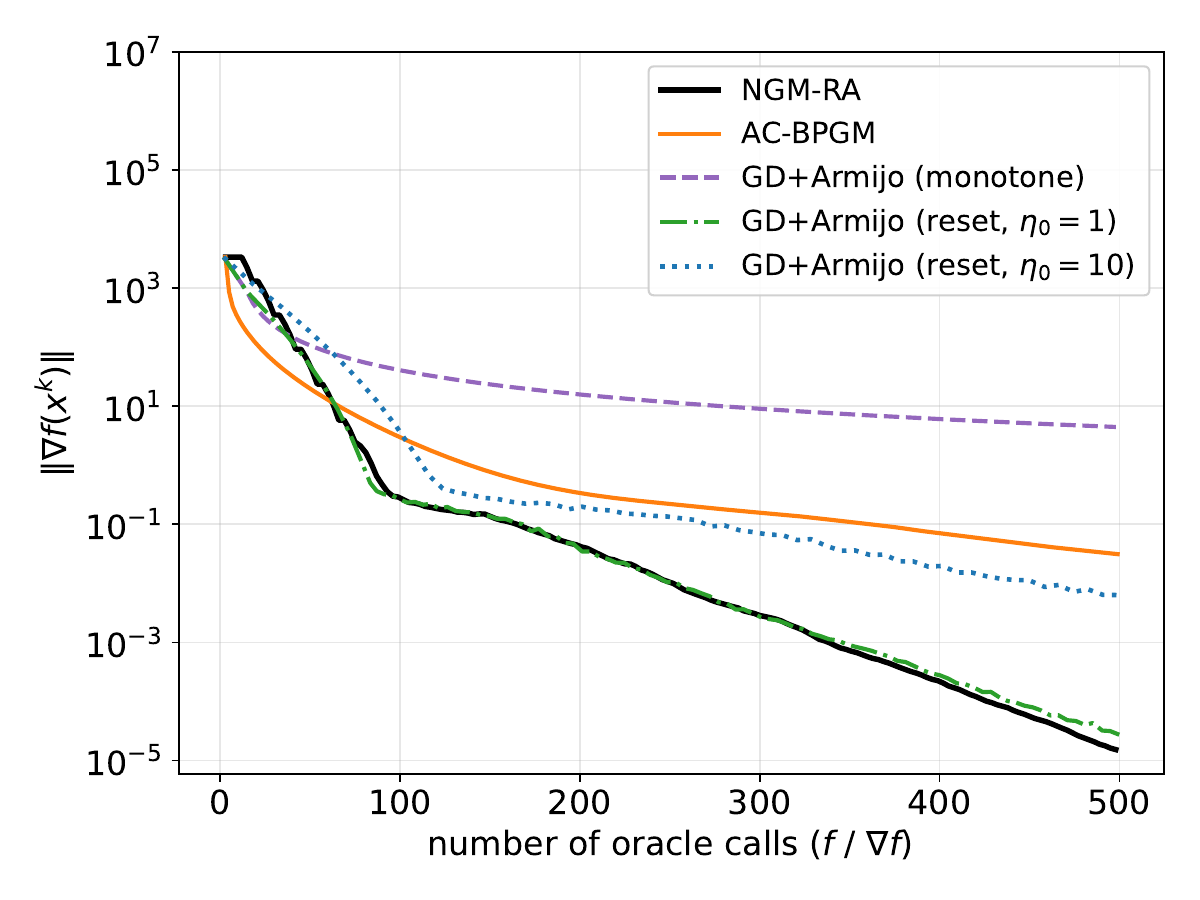}
    \caption{Convergence of Each Method versus the Number of Evaluations of $f$ and $\nabla f$ for Real Phase Retrieval}
    \label{fig:phase-ret-orac}
  \end{minipage}
\end{figure}

As in Subsection~\ref{sec-poisson}, NGM-RA converges faster than GD+Armijo with a monotone step-size update.
NGM-RA also appears to be more efficient than AC-BPGM, which we attribute to the following behavior.
When the update is small, 
the kernel function $h$ is effectively quadratic, 
so the curvature estimated by AC-BPGM through the Bregman divergence approaches the norm of the Hessian.
Moreover, since the curvature estimate in AC-BPGM increases monotonically, 
its step-size selection becomes more conservative than that of NGM-RA, which may explain the faster convergence of NGM-RA.
Turning to GD+Armijo with line search restarted from $\eta_0$ at every iteration, 
its convergence is comparable to that of NGM-RA when measured by the number of iterations.
The two methods also perform comparably in terms of the number of calls to $f$ and $\nabla f$ when $\eta_0=1$.
In contrast, GD+Armijo with $\eta_0=10$ requires more evaluations of $f$ than AC-BPGM, indicating substantially more backtracking steps.
This illustrates that the performance of line search can deteriorate when the initial step size $\eta_0$ is poorly chosen, while selecting a suitable $\eta_0$ may itself require knowledge of the problem parameters.

\section{Conclusion and Future Work}\label{conclusion-and-future-work}
In this work, we proposed \hyperref[algo-ngm-ra]{NGM-RA}, which jointly estimates the coupled parameters $(L_0,L_1)$ as a region and uses the estimated region for step-size selection.
For nonconvex $(L_0,L_1)$-smooth problems, NGM-RA finds an $\epsilon$-stationary point with an iteration complexity whose leading-order dependence on $\epsilon$ is $\mathcal{O}(L_0\Delta/\epsilon^2)$ (Subsection~\ref{sec-ncvx}),
while retaining the standard convergence guarantee in the convex setting (Subsection~\ref{sec-cvx}).
To the best of our knowledge, NGM-RA is the first method to provide these iteration-complexity guarantees simultaneously in the convex and nonconvex settings.
Our numerical experiments (Section~\ref{sec-experiments}) demonstrate the stability and efficiency gained by adaptively estimating $(L_0,L_1)$ and allowing the estimated region to evolve nonmonotonically.

On the other hand, NGM-RA has several limitations.
First, the algorithm involves rejection steps and does not allow temporary increases in the objective value.
This may cause the method to become trapped at undesirable stationary points when applied to neural network training.
Second, since NGM-RA is not accelerated, its iteration-complexity dependence on $(L_0,L_1)$ is worse than that of aGRAAL~\citep{borodich2026nesterovgraal}.
On the other hand, as observed in Section~\ref{sec-experiments}, aGRAAL appears to incur larger constant factors in practice. 
This motivates the development of an accelerated extension of NGM-RA that preserves its practical efficiency.

More broadly, the idea of adaptively estimating multiple coupled parameters as a region has the potential to extend beyond step-size selection to other optimization problems and algorithmic components.
We hope that this perspective will facilitate the development of parameter-free variants of a broader class of optimization methods and lead to practical algorithms with rigorous theoretical guarantees.

\section*{Acknowledgement}
Part of this work is supported by JST CREST JPMJCR26E3, JSPS KAKENHI JP25H01494, JP25K21158, JP26K02871, JP26K02989 and JP26K23861.

\section*{References}\label{references}
\addcontentsline{toc}{section}{References}

\renewcommand{\bibsection}{}
\bibliographystyle{plainnat}
\bibliography{ref}

\appendix

\section{Missing Proofs}\label{sec-add-proof}

\subsection{\texorpdfstring{Proofs for
Subsection~\ref{sec-prep}}{Proofs for Section~}}\label{sec-prep-proof}

\begin{tcolorbox}[mytheoremstyle]
\lemKkconsv*
\end{tcolorbox}

\begin{proof}
We prove the claim by induction.
For $k=0$, the claim follows immediately from $\tilde{\mathcal{T}}=\mathcal{T}\cap O_{\xi,r_0}\subset O_{\xi,r_0}=E_0$.

Now suppose that $\tilde{\mathcal{T}}\subset E_k$.
If iteration $k$ is successful, then $E_{k+1}\supset E_k\supset\tilde{\mathcal{T}}$, so the claim also holds at iteration $k+1$.
Suppose instead that iteration $k$ is unsuccessful. Take an arbitrary $(L_0,L_1)\in\tilde{\mathcal{T}}$. 
It suffices to show that $L_0+L_1g_k\geq s_k\gamma_k$.
By the definition of $\gamma_k$ and the induction hypothesis $\tilde{\mathcal{T}}\subset E_k$, we have $\gamma_k\leq L_0+L_1g_k$.
By the monotonicity of $\Phi$, Lemma~\ref{lem-upper}, and the definition of $s_k$, we obtain
\[
(L_0 + L_1 g_k) \Phi\left(\frac{L_0 + L_1 g_k}{\alpha \gamma_k}\right) \geq (L_0 + L_1 g_k) \Phi\left(\frac{L_1 g_k}{\alpha \gamma_k}\right) \geq M_k = s_k \gamma_k \Phi\left( \frac{s_k}{\alpha} \right).
\]

It follows that $L_0+L_1g_k\geq s_k\gamma_k$, since assuming $L_0+L_1g_k<s_k\gamma_k$ would lead to a contradiction with the monotonicity of $s\mapsto s\Phi(s/(\alpha\gamma_k))$.
\end{proof}

\begin{tcolorbox}[mytheoremstyle]
\lemdescent*
\end{tcolorbox}

\begin{proof}
Since $(L_0,L_1)\in E_k$ by Lemma~\ref{lem-Kk-consv}, the definition of $\gamma_k$ gives $\gamma_k\leq L_0+L_1g_k$.
Combining the definition of $M_k$ with the acceptance condition $M_k\leq2\beta\gamma_k$ yields the desired inequalities.
\[
\begin{aligned}
f(x^{k+1}) - f(x^k) 
&= - \frac{1}{\alpha \gamma_k} g_k^2 + \frac{1}{2 \alpha^2 \gamma_k^2} M_k g_k^2\\
&\leq - \frac{\alpha - \beta}{\alpha^2\gamma_k} g_k^2\\
&\leq - \frac{\alpha - \beta}{\alpha^2 (L_0 + L_1 g_k)} g_k^2.\\
\end{aligned}
\]
\end{proof}

\subsection{\texorpdfstring{Proofs for
Subsection~\ref{sec-num-fail}}{Proofs for Section~}}\label{sec-fail-proof}

\begin{lemma}[Region $G_k$ necessarily removed]\protect\hypertarget{lem-Dk-def-app}{}\label{lem-Dk-def-app}
Let $(l_{0,k},l_{1,k})\in E_k$ be a minimizer attaining $\gamma_k = \min_{(l_0,l_1)\in E_k} l_0+l_1g_k$. 
Suppose that iteration $k$ is unsuccessful, i.e., 
$M_k>2\beta\gamma_k$, and consider the region $G_k$ defined below, where $\mu$ is chosen to satisfy $\mu\Phi(\mu/\alpha)=2\beta$.
\[
G_k \defeq C_\xi \cap ([l_{0,k}, \mu l_{0,k}] \times [l_{1,k}, \mu l_{1,k}]).
\]
Then, $G_k\subset E_k\backslash E_{k+1}=F_{k+1}\backslash F_k$.
\end{lemma}

\begin{proof}
First, it holds that $\mu>1$. Indeed, by the strict monotonicity of $s\mapsto s\Phi(s/\alpha)$ and the inequality
$1\cdot\Phi(1/\alpha)<\Phi(1)=2(e-2)<2\le 2\beta=\mu\Phi(\mu/\alpha)$,
we obtain $\mu>1$.

We next note that $(l_{0,k},l_{1,k})\in E_k$.
Moreover, any coordinatewise enlargement $(\phi_1l_{0,k},\phi_2l_{1,k})$ with $\phi_1,\phi_2\geq1$ remains in $E_k$ whenever it belongs to $C_\xi$.
Therefore, we have $G_k\subset E_k$.

It remains to show that $G_k\cap E_{k+1}=\emptyset$. 
Since every boundary segment of $E_k$, except for the boundary of $C_\xi$, has negative slope,
it suffices to show that the upper-right vertex of the rectangular region, $(\mu l_{0,k},\mu l_{1,k})$, does not belong to $E_{k+1}$.
To this end, We show below $\mu l_{0,k} + \mu l_{1,k} g_k < s_k \gamma_k$.

Since $\mu l_{0,k}+\mu l_{1,k}g_k=\mu\gamma_k$, it is sufficient to prove that $\mu<s_k$.
By the inequality $M_k>2\beta\gamma_k$, we obtain
\[
\mu \Phi\left( \frac{\mu}{\alpha} \right) = 2 \beta < \frac{M_k}
{\gamma_k} = s_k \Phi\left( \frac{s_k}{\alpha} \right).
\]
Since $s\mapsto s\Phi(s/\alpha)$ is strictly increasing, it follows that $\mu<s_k$.
\end{proof}

\begin{tcolorbox}[mytheoremstyle]
\lemdkarea*
\end{tcolorbox}

\begin{proof}
We define
\[
\kappa_\xi \defeq \min_{\theta \in S_\xi} \int_{C_\xi \cap [\theta_0, \mu \theta_0] \times [\theta_1, \mu \theta_1]} \mathrm{d} l_0 \mathrm{d} l_1, \quad 
\mathrm{where} \quad S_\xi \defeq \{\theta \in C_\xi \, |\, \|\theta\|=1\}.
\]
First, the condition $1/\xi\Phi(1/(\alpha\xi))>2\beta=\mu\Phi(\mu/\alpha)$ in Algorithm~\ref{algo-ngm-ra}, together with the strict monotonicity of $s\mapsto s\Phi(s)$, implies that $\mu<1/\xi$.
Consequently, at least one of the following two right triangles is contained in $C_\xi$: the triangle with vertices
$(\theta_0,\theta_1),(\mu\theta_0,\theta_1),(\mu\theta_0,\mu\theta_1)$,
or the triangle with vertices
$(\theta_0,\theta_1),(\theta_0,\mu\theta_1),(\mu\theta_0,\mu\theta_1)$.
Indeed, since $\mu>1$ and $(\theta_0, \theta_1) \in C_\xi$, we have $\mu\theta_0 \ge \xi \theta_1$ and $\mu \theta_1 \ge \xi \theta_0$. Therefore, if both $(\theta_0,\mu\theta_1)$ and $(\mu\theta_0,\theta_1)$ were outside $C_\xi$, then we would have
$\xi\mu\theta_1>\theta_0$ and $\xi\mu\theta_0>\theta_1$, which together imply $\xi\mu>1$, contradicting $\mu<1/\xi$.
The above geometric observation yields the following lower bound on the integral appearing in the definition of $\kappa_\xi$:
\[\int_{C_\xi \cap [\theta_0, \mu \theta_0] \times [\theta_1, \mu \theta_1]} \mathrm{d} l_0 \mathrm{d} l_1 \ge \frac{\theta_0 \theta_1(\mu-1)^2}{2}.
\]
Moreover, since $\|\theta\|=1$ and $\theta\in C_\xi$, we may write $\theta_0=\cos\varphi$ and $\theta_1=\sin\varphi$ with $\varphi\in[\tan^{-1}(\xi),\tan^{-1}(1/\xi)]$.
Hence $\theta_0\theta_1 = \sin(2 \varphi)/2$ is minimized at $\varphi = \tan^{-1}(\xi)$, that is, at
$(\theta_0,\theta_1)=\left(\frac{1}{\sqrt{1+\xi^2}},\frac{\xi}{\sqrt{1+\xi^2}}\right)$.
Therefore, we obtain
\[
\kappa_\xi \ge \frac{\xi (\mu-1)^2}{2(1+\xi^2)}.
\]
By the definition of $\nu$ and the construction of $G_k$, we obtain
\[
\nu(G_k) = \int_{G_k} \frac{\mathrm{d} l_0 \mathrm{d} l_1}{l_0^2 + l_1^2} \geq \frac{1}{\mu^2} \int_{G_k} \frac{\mathrm{d} l_0 \mathrm{d} l_1}{l_{0,k}^2 + l_{1,k}^2} = \frac{1}{\mu^2} \int_{\frac{1}{\sqrt{l_{0,k}^2 + l_{1,k}^2}}G_k} \mathrm{d} l_0 \mathrm{d} l_1 \geq \frac{\kappa_\xi}{\mu^2} \geq \frac{\xi (\mu-1)^2}{2(1+\xi^2)\mu^2}.
\]
Combining $G_k \subset F_{k+1}\backslash F_k$ from Lemma~\ref{lem-Dk-def-app} with the above inequality yields the result.
\end{proof}

\begin{tcolorbox}[mytheoremstyle]
\lemseval*
\end{tcolorbox}

\begin{proof}
Define $C'\defeq C_\xi \cap \left\{(l_0,l_1) \relmid \sqrt{2} \tau r_0 \leq l_0 + l_1 \leq \sqrt{2} r_0\right\}$. 
Then, we have
\[
\nu(F_{k}) = \nu(\tau F_{k}) \le \nu(F_{k+1} \sqcup C') = \nu(F_{k+1}) + \nu(C').
\] 
Moreover, the measure of $C'$ can be computed as follows.
\[
\nu(C') 
= \int_{\tan^{-1}(\xi)}^{\tan^{-1}(1/\xi)} \int_{\sqrt{2}\tau r_0/(\sin \phi + \cos \phi)}^{\sqrt{2}r_0/(\sin \phi + \cos \phi)} \frac{\mathrm{d} r}{r} \mathrm{d} \phi
= \int_{\tan^{-1}(\xi)}^{\tan^{-1}(1/\xi)} \log \frac{1}{\tau} \mathrm{d} \phi \leq \frac{\pi}{2} \log \frac{1}{\tau}.
\] 
Combining these observations yields the desired bound.
\end{proof}

\begin{tcolorbox}[mytheoremstyle]
\lemEarea*
\end{tcolorbox}

\begin{proof}

For any $\theta=(\theta_0,\theta_1)\in\mathcal{E}$, either $\theta_0\le L_0$ or $\theta_1\le L_1$. Since $\theta\in C_\xi$, these two cases imply $\theta_1\le L_0/\xi$ or $\theta_0\le L_1/\xi$, respectively. 
Hence,
\[
\|\theta\|^2 = \theta_0^2 + \theta_1^2\leq \max\left\{\frac{L_0^2}{\xi^2} + L_0^2, \frac{L_1^2}{\xi^2} + L_1^2 \right\} \leq \left(\frac{1}{\xi^2} + 1\right)(L_0^2 + L_1^2) =: c_\xi^2.
\]
Therefore, it holds that $\mathcal{E}\subset\left\{\theta\relmid r_0\le\|\theta\|\le c_\xi\right\}$, and hence we obtain
\[
\nu(\mathcal{E}) \le  \int_{\tan^{-1}(\xi)}^{\tan^{-1}(1/\xi)} \int_{r_0}^{c_\xi} \frac{\mathrm{d} r}{r} \mathrm{d} \phi \leq \frac{\pi}{2} \int_{r_0}^{c_\xi} \frac{\mathrm{d} r}{r} = \frac{\pi}{2} \log \left( \frac{c_\xi}{r_0}\right),
\]
which proves the claim.
\end{proof}

\subsection{\texorpdfstring{Proofs for
Subsection~\ref{sec-cvx}}{Proofs for Section~}}\label{sec-cvx-proof}

\begin{lemma}[Monotonicity of the distance to an optimal solution]\protect\hypertarget{lem-monotone-closeness-app}{}\label{lem-monotone-closeness-app}
The sequence $R_k$ is nonincreasing.

\end{lemma}

\begin{proof}
For any successful iteration $k$, using convexity, $f(x^*)\leq f(x^{k+1})$, and the descent inequality in Lemma~\ref{lem-descent}, we obtain
\[
\begin{aligned}
R_{k+1}^2 &= \|x^{k+1} - x^*\|^2 = \|x^k-x^*\|^2 -\frac{2}{\alpha \gamma_k} \nabla f(x^k)^T(x^k-x^*) + \frac{g_k^2}{\alpha^2 \gamma_k^2}\\
&\leq R_k^2 + \frac{2}{\alpha \gamma_k} (f(x^*)-f(x^k)) + \frac{g_k^2}{\alpha^2 \gamma_k^2}\\
&\leq R_k^2 + \frac{2}{\alpha \gamma_k} (f(x^{k+1})-f(x^k)) + \frac{g_k^2}{\alpha^2 \gamma_k^2}\\
&\leq R_k^2 - \frac{2(\alpha - \beta) g_k^2}{\alpha^3 \gamma_k^2} + \frac{g_k^2}{\alpha^2 \gamma_k^2}
= R_k^2 - \frac{\alpha-2\beta}{\alpha^3 \gamma_k^2} g_k^2.
\end{aligned}
\]
Recall that Algorithm~\ref{algo-ngm-ra} assumes $\alpha>2\beta$.
\end{proof}

\begin{lemma}[Upper bound on $\Delta$]\protect\hypertarget{lem-f0-eval}{}\label{lem-f0-eval}
Suppose that Assumptions~\ref{asm-basic} and~\ref{asm-l0l1sm} hold. 
\[
\Delta \leq \frac{L_0 R^2}{2}e^{L_1 R}.
\]
\end{lemma}
\begin{proof}
By Lemma~\ref{lem-upper},
\[
f(x^0)-f(x^*)\leq \frac{L_0R^2}{2}\Phi(L_1R).
\]
Moreover,
\[
\Phi(s)
=2\sum_{l=0}^{\infty}\frac{s^l}{(l+2)!}
\leq \sum_{l=0}^{\infty}\frac{s^l}{l!}
=e^s.
\]
Combining these inequalities yields the claim.
\end{proof}

\begin{tcolorbox}[mytheoremstyle]
\thmconvoptv*
\end{tcolorbox}

\begin{proof}
By Lemma~\ref{lem-monotone-closeness-app}, $R_k$ is nonincreasing.
The remainder of the proof follows the proof of Theorem 3.2 in \citet{vankov2025optimizingl0l1smooth}.
For successful iteration $k$, by the descent inequality in Lemma~\ref{lem-descent}, the following holds for $\Delta_k\defeq f(x^k)-f^*$.

\begin{equation}\protect\phantomsection\label{eq-descent}{
\Delta_k - \Delta_{k+1} \geq \frac{\alpha - \beta}{\alpha^2} \frac{g_k^2}{L_0 + L_1 g_k} \fedeq \frac{\alpha - \beta}{\alpha^2} \psi(g_k).
}\end{equation} 
By convexity and the monotonicity of $R_k$, we further have
\[
\Delta_k \leq \nabla f(x^k)^T(x^k - x^*) \leq g_k R_k \leq g_k R.
\] 
Thus, $g_k\geq \Delta_k/R$. 
Since $\psi$ is increasing, Equation~\ref{eq-descent} can be bounded as follows.
\[
\begin{aligned}
\frac{\alpha - \beta}{\alpha^2} 
&\leq \frac{1}{\psi(g_k)} (\Delta_k - \Delta_{k+1}) \leq \frac{1}{\psi(\Delta_k/R)}(\Delta_k - \Delta_{k+1}) \\
&= \int_{\Delta_{k+1}}^{\Delta_{k}} \frac{1}{\psi(\Delta_k/R)} \d t 
\leq \int_{\Delta_{k+1}}^{\Delta_{k}} \frac{1}{\psi(t/R)} \d t\\
&= \int_{\Delta_{k+1}}^{\Delta_{k}} \frac{L_0 R^2}{t^2} + \frac{L_1 R}{t} \d t \\
&= L_0 R^2 \left( \frac{1}{\Delta_{k+1}} - \frac{1}{\Delta_k} \right) + L_1 R \log \frac{\Delta_{k}}{\Delta_{k+1}}.
\end{aligned}
\]

Summing this inequality over all successful iterations $k$ for which $\Delta_{k+1}\geq\epsilon$ yields the desired bound.
\[
(N_{\epsilon, \mathrm{opt}}-1) \frac{\alpha - \beta}{\alpha^2} \leq L_0 R^2 \frac{1}{\epsilon} + L_1 R \log \frac{\Delta}{\epsilon}.
\]

The second complexity bound follows from the upper bound on $\Delta$ established in Lemma~\ref{lem-f0-eval} below.

\[
\begin{aligned}
\frac{L_0 R^2}{\epsilon} + L_1 R \log \frac{f(x^0)-f^*}{\epsilon}
&\leq \frac{L_0 R^2}{\epsilon} + L_1 R \left(L_1 R + \log \frac{L_0 R^2}{\epsilon}\right)\\
&\leq \frac{L_0 R^2}{\epsilon} + (L_1 R)^2 + 2 L_1 R \left(\sqrt{\frac{L_0 R^2}{\epsilon}} + 1\right)\\
&\in O\left(\frac{L_0 R^2}{\epsilon} + (L_1 R)^2 + 2  \sqrt{\frac{L_0 R^2}{\epsilon} (L_1 R)^2} \right) \\
&\subset O\left(\frac{L_0 R^2}{\epsilon} + (L_1 R)^2\right). 
\end{aligned}
\] 
The final inclusion follows from the arithmetic-geometric mean inequality.
\end{proof}

\subsection{\texorpdfstring{Proofs for
Subsection~\ref{sec-ncvx}}{Proofs for Section~}}\label{sec-ncvx-proof}

\begin{tcolorbox}[mytheoremstyle]
\thmepscomp*
\end{tcolorbox}

\begin{proof}
Since $x \mapsto \frac{x^2}{L_0 +L_1 x}$ is nondecreasing for $x \ge0$, $g_k \geq \epsilon$ yields
\[
\frac{g_k^2}{L_0 + L_1 g_k} \ge \frac{\epsilon^2}{L_0 + L_1 \epsilon}
\] 

Therefore, summing the inequality in Lemma~\ref{lem-descent} over all successfull iteration $k$ such that $g_k\geq\epsilon$, gives
\[
f(x^0) - f^*   \geq N_{\epsilon, \mathrm{sta}} \frac{\alpha - \beta}{\alpha^2} \frac{\epsilon^2}{L_0 + L_1 \epsilon}.
\]

Rearranging this inequality yields the desired bound.
\[
N_{\epsilon, \mathrm{sta}} \leq \frac{\alpha^2}{\alpha - \beta} \left(\frac{L_0}{\epsilon^2} + \frac{L_1}{\epsilon}\right) (f(x^0) - f^*).
\]
\end{proof}

\subsection{\texorpdfstring{Proofs for
Section~\ref{sec-experiments}}{Proofs for Section~}}\label{proofs-for-sec-experiments}

We next show that the Poisson regression objective in Equation~\ref{eq-poisson} is $(L_0,L_1)$-smooth.

\begin{proposition}[$(L_0,L_1)$-smoothness of Poisson regression]\protect\hypertarget{prp-poisson}{}\label{prp-poisson}
Consider the following Poisson regression problem:
\begin{equation}\label{eq-poisson-app}
\min_{x \in \R^d,\, c\in \R} \frac{1}{K} \sum_{k=1}^K \left[e^{c + a_k^T x} - b_k (c + a_k^T x)\right] + \frac{\lambda}{2} \|x\|^2, \quad b_k \in \mathbb{Z}_{\geq0}, \, a_k \in \R^d.
\end{equation}

Let $f$ denote the objective function in Equation~\ref{eq-poisson-app}. Then,
 \[
\|\nabla^2 f(x,c)\| \leq  (A \overline{b} + \lambda) + A \|\nabla f(x,c)\|.
\] 
where $\overline{b}=\frac{1}{K}\sum_{k=1}^K b_k, \, A =1+\max_k \|a_k\|^2$

Thus, $f$ is $(A \overline{b} + \lambda, A)$-smooth.
\end{proposition}
\begin{proof}
First, the norm of the gradient is bounded lower as follows.
\[
\|\nabla f(x,c)\| \geq \left|\frac{\partial f}{\partial c}(x,c)\right| = \left|\frac{1}{K} \sum_k \left[e^{c+a_k^T x} - b_k\right]\right| \geq \frac{1}{K} \sum_k e^{c+a_k^T x} - \overline{b}.
\]
Using this bound, we can upper-bound the Hessian norm as
\[
\|\nabla^2 f(x,c)\| \leq \frac{1}{K}\sum_k e^{c+a_k^T x} (1+\|a_k\|^2) + \lambda \leq \frac{A}{K} \sum_k e^{c+a_k^T x} + \lambda \leq A \|\nabla f(x,c)\| + A \overline{b} + \lambda.
\]
\end{proof}

\end{document}

%% file: preamble.tex
\usepackage{amsmath,amssymb,mathtools}
\usepackage{amsthm}
\usepackage{thmtools}

\usepackage{xcolor}
\usepackage{graphicx}

\usepackage{booktabs}
\usepackage{longtable}
\usepackage{array}
\usepackage{calc}

\usepackage{caption}
\usepackage{subcaption}
\usepackage{lmodern}

\usepackage{tcolorbox}

\usepackage{float}
\usepackage{algorithm}
\usepackage[
  noEnd=false,
  indLines=false,
  italicComments=false,
  rightComments=false,
  commentColor=black,
  beginComment=//~
]{algpseudocodex}

\floatname{algorithm}{Algorithm}
\usepackage[round,authoryear]{natbib}

\usepackage{xurl}

\usepackage[
  unicode,
  colorlinks=true,
  linkcolor=blue,
  citecolor=blue,
  urlcolor=blue
]{hyperref}

\usepackage{bookmark}

\usepackage{array}
\newcolumntype{L}[1]{>{\raggedright\arraybackslash}m{#1}}
\newcolumntype{C}[1]{>{\centering\arraybackslash}m{#1}}

\usepackage[affil-it]{authblk}

%% file: macros.tex
\newcommand{\R}{\mathbb{R}}
\newcommand{\C}{\mathbb{C}}

\newcommand{\defeq}{\coloneqq}
\newcommand{\fedeq}{\eqqcolon}

\newcommand{\relmid}{\!\mathrel{}\middle|\mathrel{}\!}

\renewcommand{\d}{\mathop{}\!\mathrm{d}}

%% file: theorems.tex
\theoremstyle{plain}

\newtheorem{assumption}{Assumption}[section]

\newtheorem{lemma}{Lemma}[section]
\newtheorem{proposition}{Proposition}[section]
\newtheorem{corollary}{Corollary}[section]

\theoremstyle{definition}

\newtheorem{definition}{Definition}[section]

\theoremstyle{remark}

\AtBeginDocument{%
}

\tcbset{
  mytheoremstyle/.style={
    colback=black!5!white,
    colframe=white,
    boxrule=0pt,
    sharp corners,
    before skip=8pt,
    after skip=8pt,
    left=5pt,
    right=5pt,
    top=3pt,
    bottom=3pt
  }
}

\tcolorboxenvironment{theorem}{
  mytheoremstyle
}

\tcolorboxenvironment{lemma}{
  mytheoremstyle
}

\tcolorboxenvironment{proposition}{
  mytheoremstyle
}

\tcolorboxenvironment{corollary}{
  mytheoremstyle
}

\tcolorboxenvironment{definition}{
  mytheoremstyle
}